\documentclass{amsart}
\usepackage{graphicx}
\usepackage{natbib}
\usepackage{amsaddr}
\usepackage{longtable}

\newcommand{\inv}{^{-1}}

\newtheorem{thm}{Theorem}[section]

\newtheorem{lemma}[thm]{Lemma}
\newtheorem{defn}[thm]{Definition}
\newtheorem{cor}[thm]{Corollary}
\newtheorem{exam}[thm]{Example}

\newtheorem{conj}[thm]{Conjecture}
\newtheorem{ques}[thm]{Question}
\newtheorem{rem}[thm]{Remark}
\begin{document}
\title[Counting Condorcet domains]{Counting Condorcet domains}

\author[G. Liversidge]{Georgina Liversidge}
\address[G. Liversidge]{Department of Mathematics,\\ University of Auckland\\
{\em gliv560@aucklanduni.ac.nz}\\
{\em Feb, 2020}}

\begin{abstract}
A Condorcet domain is a collection of linear orders which satisfy an acyclic majority relation. In this paper we describe domains as collections of directed Hamilton paths. We prove that while Black's single-peaked domains are defined by their extremal paths, Arrow's single-peaked domains are not. We also introduce  domain contractions and domain extensions as well as self-paired domains, and describe some properties of these. We give a formula for the number of isomorphism classes of Arrow's single-peaked domains in terms of the number of self-paired domains, and give upper and lower bounds on this number. We also enumerate the distinct maximal Arrow's single-peaked domains for $|A|=5,6,7, 8$. Finally, we show that all of the observations in this paper can be translated to single-dipped domains, that is, Condorcet domains with complete ``never-top" conditions.
\end{abstract}
%\meaning\maketitle
\maketitle
%%%%%%%%%%%%%%%%%%
%%%%%%%%%%%%%%%%%%
%%%%%%%%%%%%%%%
\begin{section}{Introduction}

Let $A$ be a finite set and let $\mathcal{L}(A)$ be the set of all linear orders on $A$. A {\it Condorcet domain} on $A$ is a subset $D\subseteq \mathcal{L}(A)$ such that every profile composed of preferences from $D$ has an acyclic majority relation, that is, it does not contain a {\it Condorcet triple}
$$a_1\succ_{v_1} a_2\succ_{v_1} a_3,\>\>\>\> a_2\succ_{v_2} a_3\succ_{v_2} a_1,\>\>\>\> a_3\succ_{v_3} a_1\succ_{v_3} a_2,$$
as defined by M. J. Condorcet [3].
Condorcet domains have been studied extensively (see, for example [1-9], [13]), particularly  maximal Condorcet domains.
\begin{defn}
\label{defn:max}
{\em A Condorcet domain $D$ on a set $A$ is called} maximal {\em if there is no other Condorcet domain $D'$ with $D\subset D'$. In other words, if $v$ is any linear order of A such that $v\not\in D$, then $D\cup \{v\}$ contains a Condorcet triple.}
\end{defn}
It follows from Definition \ref{defn:max} that if a $D$ is maximal and $D\cup \{v\}$ does not contain a Condorcet triple, then $v\in D$.

A.K. Sen showed in [11] that if $D$ is a maximal Condorcet domain on $A$, then one of the following ``never conditions" holds for every triple $\{a,b,c\}\subseteq A$:
\begin{itemize}
\item ``Never-bottom": if $a$ is the ``never-bottom" element of the triple $\{a,b,c\}$ then the preferences $b\succ c\succ a$ and $c\succ b \succ a$ are not permitted in $D$.
\item ``Never-top": if $a$ is the ``never-top" element of the triple $\{a,b,c\}$ then the preferences $a\succ c\succ b$ and $a\succ b \succ c$ are not permitted in $D$.
\item ``Never-middle": if $a$ is the ``never-middle" element of the triple $\{a,b,c\}$ then the preferences $b\succ a\succ c$ and $c\succ a \succ b$ are not permitted in $D$.
\end{itemize}

A Condorcet domain may satisfy a mixture of these types of conditions, or there can be a single type that holds for every triple. Domains satisfying a ``never-bottom" condition on each triple are known as {\it Arrow's single-peaked domains}, after K.J. Arrow following his work in [1]. A specific type of these domains, defined below, was first described by D. Black in [2].

\begin{defn}
\label{defn:Black}
{\em A} Black's single-peaked domain {\em is a domain $D$ on a set $A$, such that there exists some ``societal axis", that is, a sequence $a_1>a_2>\dots >a_m$ of the elements of $A$, so that every order $v\in D$ has a ``peak" at some $a\in A$, with the property that for all $b,c\in A$,
\begin{itemize}
\item if $b<c\leq a$ then $c\succ_v b$ and
\item if $a\geq b>c$ then $b\succ_v c$. 
\end{itemize}
}
\end{defn}

The next two definitions, lemma, and theorem conveyed to the author by A. Slinko [12] lead to the question which is the main motivation for this paper.

\begin{defn}
{\em Let $D$ be a Condorcet domain on a set $A$. A} terminal element of $D$ {\em is an element $a\in A$ such that there exists a linear order $v\in D$ which ends with $a$.}
\end{defn}

\begin{lemma}{\em [12]}
\label{TermVertLemma}
An Arrow's single-peaked domain has at most two terminal elements. A maximal Arrow's single peaked domain has exactly two of them.
\end{lemma}

\begin{defn}
{\em A linear order which starts with one terminal element and ends at the other is called} extremal.
\end{defn}

\begin{thm}{\em [12]}
\label{ExtremalTheorem}
Any maximal Arrow's single-peaked domain on a set of size $m$ must contain exactly $2^{m-1}$ linear orders, two of which must be extremal.
\end{thm}

Note that it follows from Definition \ref{defn:Black} that a maximal Black's single-peaked Condorcet domain must contain two linear orders which are mutually reverse - that is, if $P=a_1a_2\dots a_{m-1}a_m$ is one extremal order, then $Q=a_ma_{m-1}\dots a_2a_1$ must be the other. Furthermore, a Black's single-peaked domain is uniquely determined by it's extremal orders. A. Slinko posed the following question in [12]:
\begin{ques}
\label{q:main}
Given two extremal orders $P$ and $Q$, is there a unique Arrow's single-peaked Condorcet domain containing $P$ and $Q$?
\end{ques}
This question motivated the author to study Condorcet domains, and show that the answer is, in fact, no. This then raises the follow up question: how many maximal Arrow's single-peaked domains are? Counting the number of isomorphism classes of maximal Arrow's single-peaked domains is the main focus of this paper, and we significantly reduce this problem to enumerating only self-paired domains.

First, in Section \ref{S:notation}, we outline some notation and begin describing domains as collections of directed Hamilton paths. We also introduce domain contractions and simplified domain contractions.

In Section \ref{S:conj} we show that the answer to Question \ref{q:main} is no: there may be multiple Arrow's single-peaked domains for some pairs of extremal paths. We show this by giving an example on a set of size 6. 

In Section \ref{S:cont_ext} we further explore domain contractions. We introduce domain extensions and provide some of their properties.

In Section \ref{S:iso_count} we introduce inherited permutations and provide necessary conditions for two maximal Arrow's single-peaked domains to be isomorphic. We introduce the concept of self-paired domains and show their importance in counting maximal Arrow's single-peaked domains. We also give some necessary conditions for a domain to be self-paired and give bounds for the total number of non-isomorphic maximal Arrow's single-peaked domains on a set of size $m$.

In Section \ref{S:theta_count} we enumerate the isomorphism classes of maximal Arrow's single-peaked domains with some particular inherited permutations.

In Section \ref{S:small_sets} we give some data on the isomorphism classes maximal Arrow's single peaked domains on sets of size 5, 6, 7 and 8.

In Section \ref{S:Sdip} we introduce single-dipped domains, and give some properties of these. We also give a conjecture and suggestions for future work.

Finally in Section \ref{S:append} we provide details of the distinct maximal Arrow's single-peaked domains for sets of size 5 and 6.

\end{section}

%%%%%%%%%%
%%%%%%%
%%%%%
\begin{section}{Domains as collections of Hamilton directed paths}
\label{S:notation}

Throughout this paper we use $id$ to denote the identity permutation and we use the convention of writing permutations in cyclic notation, without the use of commas. For example $(abc)$ denotes the permutation $a \mapsto b\mapsto c\mapsto a$. In contrast, we write paths with commas, so that $(a,b,c)$ denotes the path through the vertices $a,$ $b,$ and $c$, in that order. Given a path $P$, we write $(P,x)$ for the path given by appending $P$ with the vertex $x$. Similarly, we write $(x,P)$ for the path given by $P$ prefixed with $x$. Finally, we denote the position of the vertex $a$ in the path $P$ by pos$_P(a)$, where the first vertex is in position 1. 

Let $A$ be a finite set and let $V$ be a set of $|A|$ vertices each labelled with a different element of $A$. A linear order on $A$ defines a Hamilton directed path through $V$ in the obvious way. Thus a domain on $A$ may be viewed as a collection of Hamilton directed paths thorugh $V$. For simplificity, we will simply refer to this as a Hamilton directed path through $A$. Throughout this paper we distinguish between directed Hamilton paths by showing them in different colours. 

\begin{defn}
{\em Let $A$ be a finite set, and let $D$ be a collection of Hamilton directed paths through $A$. We define the} domain contraction {\em of $D$ on a subset $S\subseteq A$ to be the set $D(S)$ of Hamilton directed paths through $S$, where for each Hamilton directed path $H\in D$ we define $H'$ in $D(S)$ through $S$ by simply removing the vertices in $A\backslash S$. Let $D'(S)$ be a subset of $D(S)$ obtained by deleting any repeated paths, and call this the} simplified domain contraction {\em of $D$ on $S$.}
\end{defn}

With this definition in mind, we can redefine the Condorcet triple condition in terms of this new configuration.
\begin{lemma}
Let $A$ be a finite set and $D$ a collection of Hamilton directed paths through $A$. $D$ defines a Condorcet domain on $A$ if and only if for every subset $S\subseteq A$ of size 3, the simplified domain contraction $D'(S)$ does not contain a {\it double cycle} (that is, three paths as shown in Figure \ref{fig:DCycle}).
\begin{figure}[h]
\centering
\includegraphics[scale=0.5]{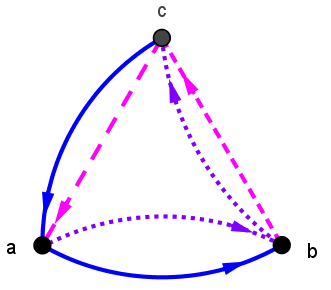}
\caption{Double cycle on three elements}
\label{fig:DCycle}%fig1
\end{figure}
\end{lemma}

It is clear to see that the pink path is the preference $b\succ c\succ a$, the blue path is $c\succ a\succ b$, and the purple path is $a\succ b\succ c$, which gives a Condorcet triple. Thus a Condorcet triple will be present if and only if a double cycle is present, as required. In fact, a simplified domain contraction on a subset of size three must be a subset of one of the path collections of the three graphs shown in Figure \ref{fig:triples}. 

Each graph corresponds to a ``never" condition on a triple:
\begin{enumerate}
\item $a$ is ``never-top" in Graph 1 of Figure \ref{fig:triples},
\item $a$ is ``never-middle" in Graph 2 of Figure \ref{fig:triples}, and
\item $a$ is ``never-bottom" in Graph 3 of Figure \ref{fig:triples}.
\end{enumerate}
\begin{figure}[h]
\centering
\includegraphics[scale=0.2]{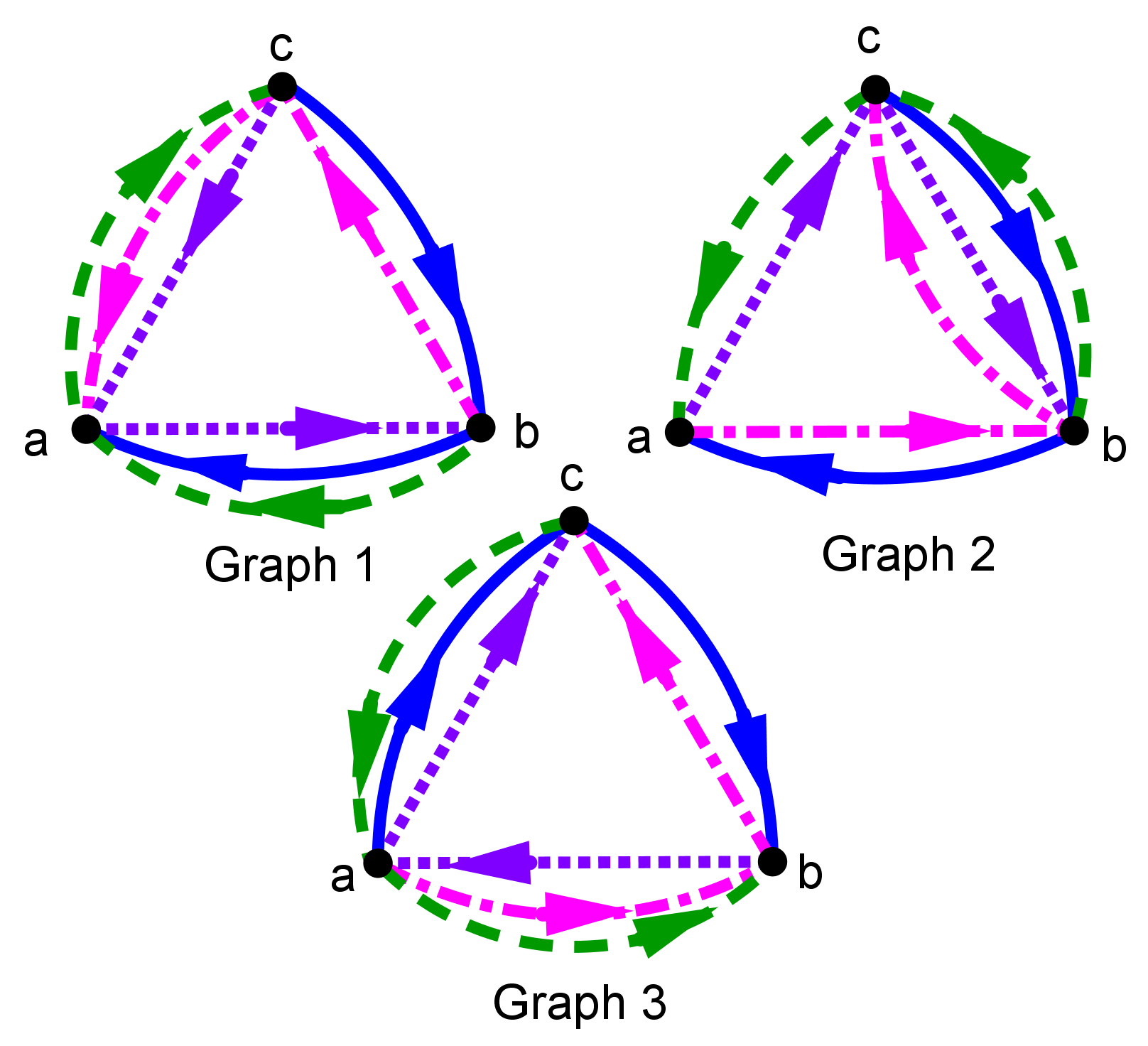}
\caption{Possible domain contractions on $\{a,b,c\}$}
\label{fig:triples}%fig2
\end{figure}

 Thus we can conclude that in order for a domain to be an Arrow's single-peaked domain, the simplified domain contraction on any three elements must be a subset of the path collection shown in Graph 3 of Figure \ref{fig:triples}. 
\end{section}

%%%%%%%%%%%%%%%
%%%%%%%%%%%%%%%%%
%%%%%%%%%
\begin{section}{Slinko's Question}
\label{S:conj}

A. Slinko in [12] posed Question \ref{q:main}: Given two extremal paths $P$ and $Q$ on a set $A$, do they uniquely define a maximal Arrow's single-peaked domain? It has been shown [10] that this is true of mutually reverse paths, that is, for Black's single-peaked domains, but, in general, the answer is no. We now prove that for a set of size 6 there may be multiple non-isomorphic maximal Arrow's single-peaked domains on a given pair of extremal paths.

\begin{exam}
\label{eg1}
{\em Let $A=\{a,b,c,d, s,f\}$ and define the Hamilton directed paths $P=(s,a,b,c,d,f)$ and $Q=(f,a,b,c,d,s)$, as shown in Figure \ref{fig:eg1}. These must be extremal paths for any domain $D$ which contains them, so any paths in $D$ must end in $s$ or $f$. 

\begin{figure}[h]
\centering
\includegraphics[scale=0.5]{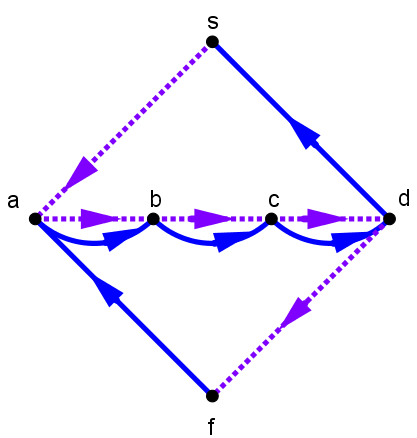}
\caption{Path $P$ shown in purple dotted lines, path $Q$ shown in blue.}
\label{fig:eg1}%fig3
\end{figure}

In the domain contraction on $\{a,b,c,d\}$, $P$ and $Q$ are equal, which gives us a level of freedom in choosing how other paths will behave on this contraction. There are, up to isomorphism, two distinct Arrow's single-peaked Condorcet domains on 4 elements. Suppose we decide that $(a,b,c,d)$ will be one extremal path, then we have two choices for the second, either the {\it twisted case} $(d,b,c,a)$ or the {\it reversed case} of $(d,c,b,a)$. In each case we have 8 possible paths on the domain contraction, and we consider how we can complete them to get paths on the original set. First, every path must end in either $s$ or $f$, so we consider those ending in $f$ and look at the possible position of $s$ in both cases. To this end, we consider the domain contractions of $P$ and $Q$, as shown in Figure \ref{fig:eg2}.
\begin{figure}[h]
\centering
\includegraphics[scale=0.3]{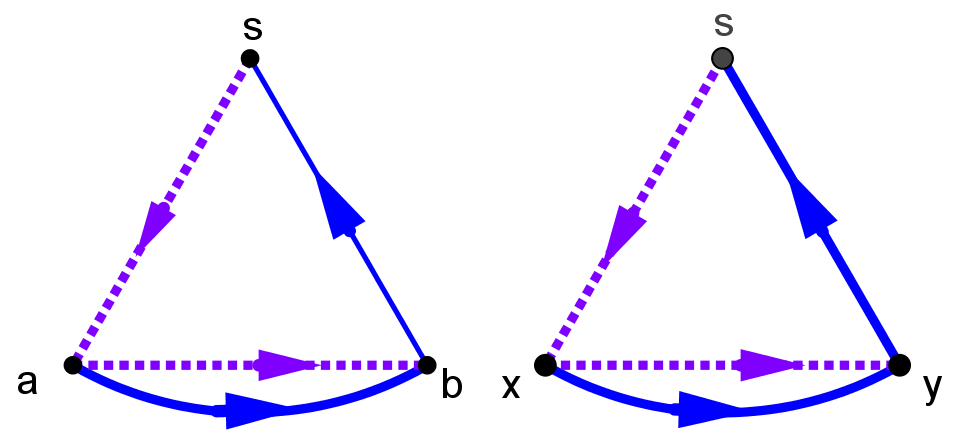}
\caption{Some domain contractions of $P$ and $Q$.}
\label{fig:eg2}%fig4
\end{figure}

Now, for a maximal Arrow's single-peaked domain, the full domain must be isomorphic to Graph 3 in Figure \ref{fig:triples}. Thus our allowed paths are $(s,a,b)$, $(a,b,s)$, $(a,s,b)$ and $(b,a,s)$, and the disallowed paths are $(s,b,a)$ and $(b,s,a)$. Similarly, for all other contractions of the form $\{s,x,y\}$ with $x,y\in \{a,b,c,d\}$, and $x$ alphabetically before $y$, the subpaths $(s,y,x)$ and $(y,s,x)$ are disallowed. With these rules in mind, we give all possible linear orders ending in $f$ for the maximal domain in both the twisted case and the reversed case, in the table below. To get the orders ending in $s$, simply swap $s$ and $f$ in the given orders.

\begin{center}
\begin{tabular}{|c|c|c|}
\hline
&Twisted case&Reversed case\\
\hline
1&$sabcdf$&$sabcdf$\\
2&$asbcdf$&$asbcdf$\\
3&$abscdf$&$abscdf$\\
4&$abcsdf$&$abcsdf$\\
5&$abcdsf$&$abcdsf$\\
6&$bascdf$&$bascdf$\\
7&$bacsdf$&$bacsdf$\\
8&$bacdsf$&$bacdsf$\\
9&$bcasdf$&$bcasdf$\\
10&$bcadsf$&$bcadsf$\\
11&$cbasdf$&$cbasdf$\\
12&$cbadsf$&$cbadsf$\\
13&$bcdasf$&$bcdasf$\\
14&$cbdasf$&$cbdasf$\\
15&$bdcasf$&$cdbasf$\\
16&$dbcasf$&$dcbasf$\\
\hline
\end{tabular}
\end{center}
Note that the orders given are the same in both cases, except for orders 15 and 16, so clearly there is no isomorphism between the twisted case and the reversed case. With the addition of the orders ending in $s$, we get a total of 32 orders in each case, as required for a maximal Arrow's single-peaked domain.}
\end{exam}

From this example we know that there may be more than one isomorphism class of maximal Arrow's single-peaked domains for a pair of extremal paths. The natural next question to ask is: how many are there for a given $m$? Or perhaps; how many are isomorphism classes are there for a given pair of extremal paths?  We explore these questions throughout the rest of the paper, and give the answers for $m\leq 8$ (in Section \ref{S:small_sets}) and some families of extremal paths (in Section \ref{S:theta_count}).
\end{section}

%%%%%%%%%%%%%%%
%%%%%%%%%%%%%%
%%%%%%%%%%%%%%%
\begin{section}{Contraction and Extension of Arrow's single-peaked domains}
\label{S:cont_ext}
We begin by giving a lemma from A. Slinko [13], transcribed into the language of domain contractions.

\begin{lemma}[13]
\label{ContractionLemma}
Let $D$ be a maximal Arrow's single-peaked domain on $A$, with terminal vertices $\{a_1, a_2\}$. If $D_i$ is the set of Hamilton directed paths in $D$ which end in $a_i$ for $i\in \{1,2\}$ then the simplified domain contraction $D'_i$ of $D_i$ on $A\backslash\{a_i\}$ is a maximal Arrow's single-peaked domain, for $i\in \{1,2\}$.  Furthermore, if
$$\hat{D_1}=\{H\in D|{\rm pos}(a_1)=m, {\rm pos}(a_2)=m-1\} \subseteq D_1$$
$${\rm and}\>\>\hat{D_2}=\{H\in D|{\rm pos}(a_2)=m, {\rm pos}(a_1)=m-1\} \subseteq D_2$$
then there exists an isomorphism $\phi$ between $\hat{D_1}$ and $\hat{D_2}$ such that $\phi(a_1)=a_2$, $\phi(a_2)=a_1$, and $\phi(a)=a$ for all $a\in A\backslash\{a_1,a_2\}$.
\end{lemma}

The following lemma is a consequence of the above.
\begin{lemma}
\label{CountingLemma}
Let $D$ be a maximal Arrow's single-peaked domain on a set of size $m$. If $a_1$ and $a_2$ are the terminal vertices of $D$ then for $2\leq j\leq m,$ and $i\in \{1,2\}$, there are $2^{j-2}$ paths in $D$ with $a_i$ in the $jth$ position, and one path with $a_i$ in the first position.
\end{lemma}
\begin{proof}
We proceed by induction on $m$, with the base case $m=3$. Here we refer back to Graph 3 of Figure \ref{fig:triples} and we concern ourselves with the terminal vertex $b$ which appears in position 3 in two paths, position 2 in one path and position 1 in the other. The same is true for the other terminal vertex $c$. Hence, the claim of Lemma \ref{CountingLemma} is true for $m=3$.
Now assume this claim is true for $m-1$, and let $D$ be a maximal Arrow's single-peaked domain on  a set $A$ of size $m$, with terminal vertices $\{a_1, a_2\}$. Without loss of generality we prove the claim for $a_1$. By Lemma \ref{ContractionLemma}, the paths ending in $a_i$ form a maximal Arrow's single peaked domain $D_i$ on $A\backslash\{a_i\}$, for $i\in\{1,2\}$. Hence there are $2^{m-2}$ paths ending in $a_1$, as required, for $j=m$. Furthermore, $a_1$ is a terminal vertex in $D'_2$, which is a domain on $m-1$ elements, and by induction, the claim holds in this domain. Hence, in $D'_2$ there are $2^{j-2}$ paths with $a_1$ in position $j$, for $2\leq j\leq m-1$, and one path with $a_1$ in position 1. However, each path is in one-to-one correspondence with a path in $D$, and this completes the proof.
\end{proof}

Note that for a given pair of extremal paths $P$ and $Q$, the number of non-equal, but possibly isomorphic, maximal Arrow's single-peaked domains with $P$ and $Q$ as extremal paths must be a power of two. This is due to the fact that for each triple the ``never-bottom" element is either set or there are two options for it. \\

\begin{defn}
{\em Let $D$ be a Condorcet domain on $A$. A} domain extension {\em of $D$ by $x$ is a Condorcet domain $E$ on $A\cup \{x\}$, such that the simplified domain contraction $E'(A)$ is equal to $D$.}
\end{defn}

\begin{lemma}
\label{L:Ext}
If $D$ is an Arrow's single peaked domain on a set $A$ of size $m$ then there are $2^{m-1}$ ways to extend $D$ to a maximal Arrow's single-peaked domain $E$ on $A\cup \{x\}$ in such a way that $x$ is a terminal vertex in $E$.
\end{lemma}
\begin{proof}
First note that this lemma does not state that the domains achieved in this process are non-isomorphic, only that they are not equal.
We proceed by induction. If $m=3$ and $D=\{abc, bac, acb, cab\}$, then $E$ can be one of the following:
\begin{enumerate}
\item $\{abcx, bacx, cabx, acbx, abxc, baxc, axbc, xabc\}$;
\item $\{abcx, bacx, cabx, acbx, abxc, baxc, bxac, xbac\}$;
\item $\{abcx, bacx, cabx, acbx, acxb, caxb, axcb, xacb\}$;
\item $\{abcx, bacx, cabx, acbx, acxb, caxb, cxab, xcab\}.$
\end{enumerate}
Now assume that for any domain on a set of size $m-1,$ there are $2^{m-2}$ ways of adding $x$ as a terminal vertex. We prove the claim for a set of size $m$.
First let $E_1$ be the set of paths on $A\cup \{x\}$ obtained from $D$ by appending $x$ to the end of each path. 
Next let $\{s,f\}$ be the terminal vertices of $D$. Our first choice is whether $s$ or $f$ should remain a terminal vertex. If we choose $f$, then from Lemma \ref{ContractionLemma} the paths ending in $f$ form a maximal Arrow's single-peaked domain on $A\backslash\{f\}$, and so by the induction hypothesis, there are $2^{m-2}$ ways of adding $x$ as a terminal vertex to this domain. Let $D_2$ be one of these extensions. Now let $E_2$ be the set of paths on $A\cup \{x\}$ obtained from $D_2$ by appending $f$ to the end of each path. Then $E=E_1\cup E_2$ gives the maximal Arrow's single-peaked domain on $A\cup \{x\}$, and there are $2\times 2^{m-2}=2^{m-1}$ ways of obtaining a domain in this way.
\end{proof}

\begin{proof}
We proceed by induction, with the base case of $m=3$. All of the possible extensions of $D=\{abc, bac, acb, cab\}$ by $x$ are given in the proof of the proceeding lemma. Each of the resulting four domains corresponds to a particular $W\in D$, such that $(x,W)$ is in the domain extension. This proves the claim for $m=3$.
Now suppose the claim is true on any set of size $m-1$. Let $D$ be the domain we wish to extend, and let $W\in D$. 
In order to create $E$ over $A\cup \{x\}$, we first append $x$ to every path in $D$, giving our first $2^m-1$ paths of $E$. We then select all paths which end in the same element as $W$, for example $f$. By Lemma \ref{ContractionLemma}, these form a maximal Arrow's single-peaked domain on $A\backslash \{f\}$. Let $W'$ be the domain contraction of $W$ on $A\backslash \{f\}$. By the induction hypothesis, this domain can be uniquely extended to a maximal Arrow's single-peaked domain on $(A\backslash \{f\})\cup \{x\}$ containing $(x,W')$. We then append $f$ to the $2^{m-1}$ paths produced by this domain extension, which completes the domain $E$ on $A\cup \{x\}$.
\end{proof}

\begin{lemma}
\label{ShuffleLemma}
Let $D$ be a maximal Arrow's single-peaked domain on a set $A$ of size $m$. Let $f$ be a terminal vertex of $D$, and let $W\in D$ be a path with $f$ in position $i$. Let $W_k$ be the path on $A$ obtained from $W$ by moving $f$ to position $k$. If $k\geq i$ then $W_k\in D$.
\end{lemma}
\begin{proof}
If $W$ ends in $f$ then the conclusion is trivial. If $f$ is the second-to-last vertex in $W$, then it follows from Lemma \ref{ContractionLemma}. 
 Now suppose $f$ is in position $2\leq i\leq m-2$. Let $b$ be the vertex directly after $f$ in $W$, and let $a$ be any vertex before $f$ in $W$, so that the domain contraction of $W$ on $\{a,b,f\}$ is equal to $(a,f,b)$. Then the terminal vertices of this contraction must be $f$ and $b$. Thus $W_{i+1}$ obtained from $W$ by swapping $f$ and $b$ will be in $D$, since $W_{i+1}'$ will be equal to $W'$ on all other contractions. The proof now follows by induction for all $2\leq i \leq m-2$.
Finally, if $W$ begins with $(f, a)$ for some $a\in A$, then for all $b\in A\backslash \{a,f\}$, the domain contraction of $W$ on $\{a,b,f\}$ is $(f,a,b)$. Since $f$ is a terminal vertex, the ``never-bottom" element of this triple is $a$. Thus, $W_2$ will be in $D$, since it will be equal to $(a,f,b)$ on any domain contraction of this form $\{a,b,f\}$, and equal to $W$ on the domain contration of any other triple. This completes the proof.
\end{proof}

\end{section}

%%%%%%%%%%%%%%%%%%%
%%%%%%%%%%%%%%%%%%%%%
%%%%%%%%%%%%%%%%%%%
\begin{section}{Isomorphisms and counting}
\label{S:iso_count}

Example \ref{eg1} outlined in Section \ref{S:conj} shows that extremal paths do not define maximal Arrow's single-peaked domains. However, since the extremal paths of a maximal Arrow's single-peaked domain provide some characterisation of its isomorphism class, it is still worth considering extremal paths in an attempt to count the number of isomorphism classes. The following Lemma was given by A. Slinko in [12].

\begin{lemma} 
\label{L:Slinko}
[12] Let $A$ be a set of size $m$, and let $P=(s,a_1,a_2,\dots, a_{m-2},f)$ be a Hamilton directed path on $A$ and let $\theta\in S_{A}$. If $Q=\theta(P)$ and $Q'=\theta\inv(P)$ then the pair $(P,Q)$ is isomorphic to the pair $(P, Q')$.
\end{lemma}

The proof was not given by Slinko in [12], but it easy to see that $\theta$ sends $P$ to $Q$ and $Q'$ to $P$, and the proof follows. Note that if $\theta^2=id$ then $\theta$ is an automorphism.\\
Thus we may consider possible pairs of extremal paths on a set $A$ of size $m$, based on an arbitrary directed Hamilton path $P=(s,a_1,a_2,\dots, a_{m-2},f)$ and the permutations $\theta=(sf)\sigma$ with $\sigma\in S_{A\backslash\{s,f\}}$. Lemma \ref{L:Slinko} allows us to exclude one such permutation from each pair $\{\sigma, \sigma\inv\},$ for which $\sigma\inv\neq \sigma$.

\begin{defn}
{\em Let $D$ be a maximal Arrow's single-peaked domain with $P=(s, a_1, a_2, \dots, a_{m-2}, f)$ and 
 $Q=(f,\sigma(a_1),\sigma(a_2),\dots, \sigma(a_{m-2}),s)$ as extremal paths. We say that the permutation formed by the composition of $(sf)$ with $\sigma$ is the} inherited permutation {\em of $D$, and denote it by $\theta_D$. We may also refer to $\sigma$ as $\sigma_D$.}
\end{defn}
 Note that an inherited permutation must swap the terminal vertices of the domain. This will be assumed throughout this paper. 

\begin{lemma}
\label{isoLem}
If $D_1$ and $D_2$ are maximal Arrow's single-peaked domains with $P$ and $Q=\theta(P)$ as extremal paths, such that $D_1\neq D_2$ then if $D_1$ and $D_2$ are isomorphic then $\theta$ is the isomorphism between the two. Consequently if $\theta$ is not of order 2 then $D_1$ and $D_2$ cannot be isomorphic.
\end{lemma}

\begin{proof}
First note that any isomorphism between $D_1$ and $D_2$ must either swap or fix $P$ and $Q$. Hence the only possible isomorphisms are the identity and $\theta$. However, the identity is excluded by the assertion that $D_1\neq D_2$. Hence $D_1$ and $D_2$ are isomorphic if and only if $\theta$ is the isomorphism between them.
 Now suppose $\theta^2$ is not the identity. Then $\theta(P)=Q$ and $\theta(Q)=\theta^2(P)\neq P$. Hence $\theta$ can not be an isomorphism between $D_1$ and $D_2$ as it neither fixes nor swaps $P$ and $Q$. Therefore $\theta^2$ must be the identity if $D_1$ and $D_2$ are isomorphic.
\end{proof}

\begin{defn}
{\em Let $D$ be a maximal Arrow's single-peaked domain with inherited permutation $\theta$. If $\theta$ maps $D$ to itself we say that $D$ is} self-paired.
\end{defn}
The next lemma follows directly from Lemma \ref{isoLem}.
\begin{lemma}
\label{SPLemma}
If $D$ is a self-paired maximal Arrow's single-peaked domain with inheirited permutation $\theta$, then $\theta$ must have order 2.
\end{lemma}

\begin{lemma}
\label{selfPLemma}
Let $D$ be a maximal Arrow's single-peaked domain on a set $A$ with inherited permutation $\theta$. Then $D$ is self-paired if and only if for each triple $T\subseteq A$, if $a$ is the ``never-bottom" element of a triple $T$, then $\theta(a)$ is the ``never-bottom" element of $\theta(T)$. 
\end{lemma}

\begin{proof}
By definition, $D$ is self-paired if and only if $\theta(W)\in D$ for each $W\in D$. It follows that if $W$ terminates at $b\in B$ on some domain contraction on $B\subseteq A$, then the domain contraction on $\theta(B)$ then $\theta(W)$ terminates at $\theta(b)$. Thus $\theta(b)$ must be a terminal vertex of the domain contraction on $\theta(B)$ in order for $D$ to be self-paired. This occurs for all $W\in D$ if and only if it occurs on every triple. 
\end{proof}
\begin{rem}
\label{Remark1}
{\em Let $S\subseteq A$ be the maximal subset of $A$ such that $\theta$ fixes every element in $A$. If $T$ is a triple in $A$ then $\theta$ fixes every element of $T$. Hence the ``never-bottom" elements of $T$ and $\theta(T)$ coincide. Furthermore, for triples containing both $s$ and $f$ and triples $T$ such that $T\subseteq (S\cup \{s,f\})$ the ``never-bottom of $T$ and $\theta(T)$ coincide.  Hence, the problem is reduced to checking triples which contain at least one element of $A\backslash (S\cup \{s,f\})$.
Furthermore, suppose that $D$ is self-paired, and that $b=\theta(a)$ where $a\in A\backslash S$. Let $T=\{a,b,x\}$. Then either $x$ is the ``never-bottom" element of $T$, or $\theta(x)\neq x$.}
\end{rem}
\begin{cor}
\label{cor:SP}
Let $D$ be a maximal Arrows single-peaked domain on a set $A$ with extemal paths $P$ and $\theta(P)$. If there exist $x, x', a\in A$ such that $\theta(x)=x'$, $\theta(a)=a$ and $P(\{x,x',a\})=(x, x', a)$ then $D$ is not self-paired.
\end{cor}
Note that if $\theta$ does not have such a pair there is at least one self-paired maximal Arrow's single-peaked domain with $P$ and $\theta(P)$ as extremal paths but there may also be non-self-paired maximal Arrow's single-peaked domains with $P$ and $\theta(P)$ as extremal paths. We now show the importance of self-paired domains with the main theorem of this paper.

\begin{thm}
\label{Thm:main}
Let $\mathcal{N}(m)$ denote the number of isomorphism classes of maximal Arrow's single-peaked domains on $m$ elements. Let $\mathcal{P}(m)$ denote the total number of all maximal Arrow's single-peaked domains on a set of size $m$, without reduction under isomorphism, but which all share a common extremal path. Let $\mathcal{SP}(m)$ denote the number of isomorphism classes of self-paired maximal Arrow's single-peaked domains on a set of size $m$. Then $\mathcal{P}(3)=\mathcal{N}(3)=\mathcal{SP}(3)=1,$ and for all $m\geq 4$ the following hold:
\begin{enumerate}
\item $\mathcal{P}(m)=2^{m-3}\mathcal{P}(m-1)$,
\item $\mathcal{N}(m)=\frac{1}2(\mathcal{P}(m)+\mathcal{SP}(m))$, and
\item $\frac{1}2\mathcal{P}(m)\leq \mathcal{N}(m)\leq \mathcal{P}(m).$
\end{enumerate}

\end{thm}
\begin{proof}
Note that $\mathcal{SP}(m)$ also denotes the number of of self-paired maximal Arrow's single-paired domains on $m$ elements which share an extremal path.
First we prove $\mathcal{P}(3)=\mathcal{N}(3)=\mathcal{SP}(3)=1$. Take the set $A=\{a,b,c\}$. Consider the maximal Arrow's single-peaked domains containing the path $(a,b,c)$ as an extremal path. For these, $a$ and $c$ must be the terminal vertices. But there is only one maximal Arrow's single-peaked domain for $A$ with $a$ and $c$ as terminal vertices, namely $\{(a,b,c), (c,b,a), (b,c,a), (b,a,c)\}$. Hence $\mathcal{P}(3)=1$, and since there is only one, $\mathcal{N}(3)=1$. Finally this domain is self-paired, so $\mathcal{SP}(3)=1$.\\

Next we prove \ref{Thm:main}.$(1)$. Let $S$ be the set of $\mathcal{P}(m-1)$ maximal Arrow's single-peaked domains on a set $A$, of size $m-1$, which all have a common extremal path $P=(f,a_1,\dots, a_{m-2},s)$. 
By Lemma \ref{L:Ext}, each domain in $S$ can be extended to a maximal Arrow's single-peaked domain on $A\cup \{x\}$, in a total of $2^{m-2}$ ways. Half of these domains will have terminal vertices $\{x,s\}$, and the other half will have terminal vertices $\{x,f\}$. Let $S_f$ be the half which have $f$ as a terminal vertex. From the construction in the proof of Lemma \ref{L:Ext}, the path $(P, x)$ must be common to all domains in $S_f$. Furthermore, $(P,x)$ begins with $f$ and ends with $x$, and therefore, it is an extremal path. Thus we have $\mathcal{P}(m)\geq |S_f|= 2^{m-3}\mathcal{P}(m-1)$.\\
Now suppose $D$ is a maximal Arrow's single-peaked domain on $A\cup \{x\},$ with $(P,x)$ as an extremal path. Then $P$ is an extremal path in the simplified domain contraction $D'(A)$. Thus we have $D'(A)\in S$, and hence $D$ is in $S_f$ and $\mathcal{P}(m)\leq |S_f|=2^{m-3}\mathcal{P}(m-1)$. Finally, this implies that $\mathcal{P}(m)=2^{m-3}\mathcal{P}(m-1)$.\\

Now we prove \ref{Thm:main}.$(2)$. Let $S$ be the set of $\mathcal{P}(m)$ maximal Arrow's single-peaked domains on a set $A$, of size $m$, with common path $P$. Let $B_1=\{D\in S\>\mid\>\theta_D(D)=D\}$ and $B_2=\{D\in S\>\mid\>\theta_D(D)\neq D\}$. Note that $B_1$ and $B_2$ partition $S,$ and that $\mathcal{N}(m)=|B_1|+\frac{1}2|B_2|=\mathcal{SP}(m)+\frac{1}2(\mathcal{P}(m)-\mathcal{SP}(m))=\frac{1}2(\mathcal{P}(m)+\mathcal{SP}(m)),$ as required.\\

Finally \ref{Thm:main}.$(3)$ follows directly from \ref{Thm:main}.$(2)$ and the fact that $0\leq \mathcal{SP}(m)\leq \mathcal{P}(m)$.
\end{proof}

The above theorem reduces the problem of enumerating isomorphism classes of maximal Arrow's single-peaked domains to enumerating self-paired maximal Arrow's single-peaked domains. Furthermore, by Lemma \ref{SPLemma}, we need only consider maximal Arrow's single-peaked domains with inherited permutations of order 2. Corollary \ref{cor:SP} further reduces the number of permutation which need to be considered.
\end{section}

%%%%%%%%%
%%%%%%%%
%%%%%%%%%
\begin{section}{Some particular inherited permutations}
\label{S:theta_count}

\begin{lemma}
\label{idLem1}
Let $D$ be a maximal Arrow's single peaked domain on a set $A$ of size $m$. Let $\theta_D=(sf)$, where $s$ and $f$ are the terminal vertices of $D$. If $W\in D$ with ${\rm pos}_W(f)=i\leq m-1$ then there exists some $W'\in D$ with ${\rm pos}_{W'}(a)={\rm pos}_W(a)$ for all $a\in A\backslash\{s,f\}$, and ${\rm pos}_{W'}(s)=i$ and ${\rm pos}_{W'}(f)=m$.
\end{lemma}
\begin{proof}
The proof follows directly from Remark \ref{Remark1} since the fixed points of $\theta$ are $S=A\backslash \{s,f\}$ and thus, $A\backslash (S\cup \{s,f\})=\emptyset.$
\end{proof}

\begin{thm}
\label{idThm}
The number of non-isomorphic maximal Arrow's single-peaked domains on a set of size $m$, with $\sigma_D=id$ is equal to 1 for $m=3$, and $\mathcal{P}(m-1)$ for all $m\geq 4$. All such domains are self-paired.
\end{thm}
\begin{proof}
First note that $\sigma_D=id$ in any maximal Arrow's single-peaked domain $D$ on a set of size 3. Thus for $m=3$ we have $\mathcal{N}(3)=1$ domains.
Now, let $S$ be the set of $\mathcal{P}(m-1)$ domains on a set $A$, of size $m-1$, all with $(s,P)$ as an extremal path. By Lemma \ref{ShuffleLemma}, each domain in $S$ contains the path $(P,s)$, and  by Lemma \ref{L:UniqExt}, each domain in $S$ can be uniquely extended to a domain on $A\cup\{f\}$ with $(s,P,f)$ and $(f,P,s)$ as extremal paths. Thus, there are at most $\mathcal{P}(m-1)$ domains with $(s,P,f)$ and $(f,P,s)$ as extremal paths.
Let $D$ be a domain with $(s,P,f)$ and $(f,P,s)$ as extremal paths. Then the simplified domain contraction $D'(A)$ must be in $S$. Thus, for $m\geq 4$ there are at least $\mathcal{P}(m-1)$ domains, and it follows that there are exactly $\mathcal{P}(m-1)$ with $(s,P,f)$ and $(f,P,s)$ as extremal paths.
The fact that these domains are self-paired follows directly from Lemma \ref{idLem1} and hence the domains with $(s,P,f)$ and $(f,P,s)$ as extremal paths are non-isomorphic.
\end{proof}

At this stage we have reduced the problem of counting isomorphism classes of maximal Arrow's single-peaked domains to counting self-paired maximal Arrow's single-peaked domains $D$ which have $\sigma_D$ of order 2. 

\begin{lemma}
\label{L:Shuff2}
Let $D$ be a maximal Arrow's single-peaked domain. If $P=(s,a_1,\dots, a_{m-2},f)$ is an extremal path of $D$ then $D$ contains the path $P'=(a_2,a_1,a_3,\dots, a_{m-2},f,s)$.
\end{lemma}
\begin{proof}
On any contraction on a set containing $s$ and/or $f$, the path $P'$ will end at one of the two, which is allowed. Furthermore, on any contraction on a set $T\subseteq A\backslash \{s,f\}$, of size at least 3, the path $P'$ will terminate at the same vertex as $P$, which must be allowed. Thus $P'$ satisfies the ``never-bottom" conditions on all triples, and must be in $D$, as claimed.
\end{proof}

\begin{thm}
Let $\theta=(sf)\sigma$, such that $\sigma=(a_1 a_{m-2})$ or $\sigma=(a_2 a_{m-2})$. The number of maximal Arrow's single-peaked domains with\\
 $P=(s, a_1,\dots, a_{m-2},f)$ and $Q=\theta(P)$ as extremal paths is $\mathcal{P}(m-3)$. Moreover, all such domains are self-paired.
\end{thm}
\begin{proof}
Let $A=\{s,f,a_1, \dots , a_{m-2}\}$, and let $S$ be the set of maximal Arrow's single-peaked domains on $A\backslash\{s,f\}$ with extremal paths $P_1=(a_1,\dots, a_{m-2})$ and $Q_1=(a_{m-2},a_2, a_3,\dots, a_{m-3},a_1)$.  By Theorem \ref{idThm}, $|S|=\mathcal{P}(m-3)$. 
By Lemma \ref{L:UniqExt}, any domain in $S$ can be uniquely extended twice to a domain $D_1$ on $A$ with $P=(s,a_1,\dots, a_{m-2},f)$ and $Q=(f,a_{m-2},a_2,\dots, a_{m-3},a_1,s)$ as extremal paths.
By Lemma \ref{ShuffleLemma}, each $D\in S$ also contains\\
 $P_2=(a_2,a_1,a_3,\dots, a_{m-2})$ and $Q_2=(a_2,a_{m-2},a_3,\dots, a_{m-3},a_1)$, so each domain in $S$ can also be extended twice to $D_2$ with\\
 $P'=(s,a_2,a_1,a_3,\dots, a_{m-2},f)$ and $Q'=(f,a_2,a_{m-2}, a_3,\dots, a_{m-3},a_1,s)$ as extremal paths by Lemma \ref{L:UniqExt}. Note that with relabelling this gives $P$ and $\theta(P)$ with $\theta=(sf)(a_2 a_{m-2})$. Clearly, any domain with $P$ and $Q$  or $P'$ and $Q'$ as extremal paths will be in $S$, so there are exactly $\mathcal{P}(m-3)$ of each type of domain.
 Now it remains to be shown that all such domains are self-paired. First we consider the domain $D_1$ with $P$ and $Q$ as extremal paths. By Lemma \ref{selfPLemma}, since we have $\sigma=(a_1 a_{m-2}),$ we must consider triples containing $a_1$ and/or $a_{m-2}$. Let $x,y\in A\backslash\{s,f,a_1, a_{m-2}\}$ such that $x$ is before $y$ in $P$. The table below gives details of these triples, showing that each satisfies the conditions of Lemma \ref{selfPLemma}.\\

\begin{tabular}{|c|c|c|c|}
\hline
$T$&$P$&$Q$&never bottom\\ 
\hline
$\{a_1, x,y\}$&$(a_1,x,y)$&$(x,y,a_1)$&$x$\\
$\{a_{m-2}, x,y\}$&$(x,y,a_{m-2})$&$(a_{m-2},x,y)$&$x$\\
\hline

$\{a_1, s,x\}$&$(s,a_1,x)$&$(x,a_1,s)$&$a_1$\\
$\{a_{m-2}, f,x\}$
&$(x,a_{m-2},f)$
&$(f,a_{m-2},x)$&$a_{m-2}$\\
\hline

$\{a_1, f,x\}$
&$(a_1,x,f)$
&$(f,x,a_1)$&$x$\\
$\{a_{m-2}, s,x\}$
&$(s,x,a_{m-2})$&
$(a_{m-2},x,s)$&$x$\\
\hline

$\{a_1,a_{m-2}, s\}$&$(s,a_1,a_{m-2})$&$(a_{m-2},a_1,s)$&$a_1$\\
$\{a_1,a_{m-2}, f\}$&$(a_1,a_{m-2},f)$&$(f,a_{m-2},a_1)$&$a_{m-2}$\\
\hline

$\{a_1,a_{m-2}, x\}$&$(a_1,x,a_{m-2})$&$(a_{m-2},x,a_1)$&$x$\\
\hline
\end{tabular}\\

By Lemma \ref{selfPLemma}, if follows that $D_1$ is self-paired. Next, for $D_2$ we have the same permutation, and we get the same table as above, except that we must consider $a_2$ separately from $x$ and $y$. For these triples we get the following:\\

\begin{tabular}{|c|c|c|c|}
\hline
$T$&$P'$&$Q'$&never bottom\\ 
\hline
$\{a_1, a_2,x\}$&$(a_2,a_1,x)$&$(a_2,x,a_1)$&$a_2$\\
$\{a_{m-2}, a_2,x\}$&$(a_2,x,a_{m-2})$&$(a_2,a_{m-2},x)$&$a_2$\\
\hline

$\{a_1, s,a_2\}$&$(s,a_2,a_1)$&$(a_2,a_1,s)$&$a_2$\\
$\{a_{m-2}, f,a_2\}$
&$(a_2,a_{m-2},f)$
&$(f,a_2,a_{m-2})$&$a_2$\\
\hline

$\{a_1, f,a_2\}$
&$(a_2,a_1,f)$
&$(f,a_2,a_1)$&$a_2$\\
$\{a_{m-2}, s,a_2\}$
&$(s,a_2,a_{m-2})$&
$(a_2,a_{m-2},s)$&$a_2$\\
\hline

$\{a_1,a_{m-2}, a_2\}$&$(a_2,a_1,a_{m-2})$&$(a_2,a_{m-2},a_1)$&$a_2$\\
\hline
\end{tabular}\\

Thus $D_2$ is also self-paired, as required.
\end{proof}
\end{section}

%%%%%%%%%%%%%%%%%%%%
%%%%%%%%%%%%%%%%%%%%%
%%%%%%%%%%%%%%%%%%

\begin{section}{Isomorphism classes for small sets}
\label{S:small_sets}

The table below gives $\mathcal{P}(m)$, $\mathcal{SP}(m)$ and $\mathcal{N}(m)$, as defined in Theorem \ref{Thm:main}, for $m\in\{3,4,5,6,7,8\}$. Note that by Theorem \ref{idThm}, the number of isomorphism clasess of maximal Arrow's single-peaked domains with inherited permutation $(sf)$ is $\mathcal{P}(m-1)$, and so $\mathcal{SP}(m)-\mathcal{P}(m-1)$ gives the number of self-paired maximal Arrow's single-peaked domains with $\sigma_D$ of order 2. This number is all that needs to befound in order to calculate $\mathcal{N}(m)$, by Theorem \ref{Thm:main}.\\

\begin{tabular}{|c|c|c|c|c|}
 \hline
$m$ & $\mathcal{P}(m)$ &$\mathcal{SP}(m)-\mathcal{P}(m-1)$& $\mathcal{SP}(m)$ & $\mathcal{N}(m)$ \\
 \hline
3&1&0&1&1\\
4&2&1&2&2\\
5&$2^3$&2&4&6\\
6&$2^6$&$2^3$&16&40\\
7&$2^{10}$&$2^5$&98&560\\
8&$2^{15}$&$2^8$&$1280$&17024\\
\hline
\end{tabular}\\

We expand on this briefly in the tables below, which give the numbers of self-paired domains with terminal vertices $\{s,f\}$ and inherited permutation $(sf)\sigma$, for given permutations $\sigma$ of order 1 or 2.
For more details on isomorphism classes of maximal Arrow's single-peaked domains on sets of size 5 and 6, see Section \ref{S:append}.\\

\begin{tabular}{|c|c|}
\hline
 \multicolumn{2}{|c|}{$|A|=5$} \\
 \hline
$\sigma$&Self-paired\\
\hline
$(ac)$&1\\
$(bc)$&1\\
$(ab)$&0\\
\hline
id&2\\
\hline
Total&4\\
\hline
\hline
 \multicolumn{2}{|c|}{$|A|=6$} \\
 \hline
$\sigma$&Self-paired\\
\hline
$(ad)$&1\\
$(bd)$&1\\
$(ac)$&0\\
$(bc)$&0\\
$(cd)$&2\\
$(ab)$&0\\
\hline
$(ad)(bc)$&1\\
$(ac)(bd)$&1\\
$(ab)(cd)$&2\\
\hline
id&8\\
\hline
Total&16\\
\hline
\end{tabular}
\begin{tabular}{|c|c|}
\hline
 \multicolumn{2}{|c|}{$|A|=7$} \\
 \hline
$\sigma$&Self-paired\\
\hline
$(ae)$&2\\
$(be)$&2\\
$(ad)$&0\\
$(bd)$&0\\
$(ce)$&4\\
$(ac)$&0\\
$(bc)$&0\\
$(cd)$&0\\
$(de)$&8\\
$(ab)$&0\\
\hline
 \multicolumn{2}{c}{} \\
 \multicolumn{2}{c}{} \\
 \multicolumn{2}{c}{} \\
 \multicolumn{2}{c}{} \\
 \multicolumn{2}{c}{} \\
 \multicolumn{2}{c}{} \\
 \multicolumn{2}{c}{} \\
 \multicolumn{2}{c}{} \\
\end{tabular}
\begin{tabular}{|c|c|}
\hline
 \multicolumn{2}{|c|}{$|A|=7$} \\
 \hline
$\sigma$&Self-paired\\
\hline
$(ad)(be)$&1\\
$(ae)(bd)$&1\\
$(ae)(cd)$&1\\
$(be)(cd)$&1\\
$(ac)(be)$&0\\
$(ae)(bc)$&0\\
$(ad)(ce)$&2\\
$(bd)(ce)$&2\\
$(ac)(bd)$&0\\
$(ad)(bc)$&0\\
$(ab)(ce)$&0\\
$(ac)(de)$&4\\
$(bc)(de)$&4\\
$(ab)(cd)$&0\\
$(ab)(de)$&0\\
\hline
id&64\\
\hline
Total&96\\
\hline
 \multicolumn{2}{c}{} \\
\end{tabular}

\begin{tabular}{|c|c|}
\hline
 \multicolumn{2}{|c|}{$|A|=8$} \\
 \hline
$\sigma$&Self-paired\\
\hline
$(af)$&8\\
$(bf)$&8\\
$(ac)$&0\\
$(bc)$&0\\
$(ad)$&0\\
$(bd)$&0\\
$(ae)$&0\\
$(be)$&0\\
$(cd)$&0\\
$(de)$&0\\
$(ce)$&0\\
$(cf)$&16\\
$(df)$&32\\
$(ef)$&64\\
$(ab)$&0\\
\hline
$(ab)(cd)$&0\\
$(ab)(ce)$&0\\
$(ab)(cf)$&0\\
$(ab)(de)$&0\\
$(ab)(df)$&0\\
$(ab)(ef)$&0\\
$(ac)(bd)$&0\\
$(ac)(be)$&0\\
$(ac)(bf)$&0\\
$(ac)(de)$&0\\
$(ac)(df)$&0\\
\hline
\end{tabular}
\begin{tabular}{|c|c|}
\hline
 \multicolumn{2}{|c|}{$|A|=8$} \\
 \hline
$\sigma$&Self-paired\\
\hline
$(ac)(ef)$&0\\
$(bc)(de)$&0\\
$(bc)(df)$&0\\
$(bc)(ef)$&0\\
$(ad)(bc)$&0\\
$(ad)(be)$&0\\
$(ad)(bf)$&0\\
$(ad)(ce)$&0\\
$(ad)(cf)$&0\\
$(ad)(ef)$&8\\
$(bd)(ce)$&0\\
$(bd)(cf)$&0\\
$(bd)(ef)$&8\\
$(ae)(bc)$&0\\
$(ae)(bd)$&0\\
$(ae)(bf)$&1\\
$(ae)(cd)$&0\\
$(ae)(cf)$&2\\
$(ae)(df)$&4\\
$(be)(cd)$&0\\
$(be)(cf)$&2\\
$(be)(df)$&4\\
$(af)(bc)$&0\\
$(af)(bd)$&0\\
$(af)(be)$&1\\
$(af)(cd)$&0\\
\hline
\end{tabular}
\begin{tabular}{|c|c|}
\hline
 \multicolumn{2}{|c|}{$|A|=8$} \\
 \hline
$\sigma$&Self-paired\\
\hline
$(af)(ce)$&1\\
$(af)(de)$&2\\
$(bf)(cd)$&0\\
$(bf)(ce)$&1\\
$(bf)(de)$&2\\
$(cd)(ef)$&16\\
$(ce)(df)$&8\\
$(cf)(de)$&4\\
\hline
$(ab)(cd)(ef)$&16\\
$(ab)(ce)(df)$&8\\
$(ab)(cf)(de)$&4\\
$(ac)(bd)(ef)$&8\\
$(ac)(be)(df)$&4\\
$(ac)(bf)(de)$&2\\
$(ad)(bc)(ef)$&8\\
$(ae)(bc)(df)$&4\\
$(af)(bc)(de)$&2\\
$(ad)(be)(cf)$&2\\
$(ad)(bf)(ce)$&1\\
$(ae)(bd)(cf)$&2\\
$(ae)(bf)(cd)$&1\\
$(af)(bd)(ce)$&1\\
$(af)(be)(cd)$&1\\
\hline
id&1024\\
\hline
Total&1280\\
\hline
 \multicolumn{2}{c}{} \\
\end{tabular}\\
\end{section}

%%%%%%%%%%%%%%
%%%%%%%%%%%%%%
%%%%%%%%%%%%%%%%%
\begin{section}{Single-Dipped Domains and future directions.}
\label{S:Sdip}
\begin{defn}
{\em Let $D$ be a Condorcet domain on a set $A$. We say that $D$ is a} single-dipped domain {\em if $D$ has a ``never-top" element on every triple $T\subseteq A$.}
\end{defn}

\begin{thm}
\label{L:peaktodip}
Let $S_1$ be the set of Arrow's single-peaked domains on a set $A$ and $S_2$ be the set of single-dipped domains on $A$. The elements of $S_1$ are in one-to-one correspondence with the elements of $S_2$.
\end{thm}
\begin{proof}
 Let $\phi : S_1\to S_2$ be a function which acts on $D\in S_1$ by reversing each path in $D$. Since $D$ is an Arrow's single-peaked domain, it has a ``never-bottom" element in every triple $T\subseteq A$. Thus $\phi(D)$ must have a ``never-top" element on $T$. Hence $\phi(D)$ is a single-dipped domain. Clearly $\phi$ is self-inverse and is therefore a one-to-one correspondence, as required.
\end{proof}
The following is a direct consequence of Theorem \ref{L:peaktodip}. 
\begin{cor}
The number of maximal Arrow's single-peaked domains on a set $A$ is equal to the number of maximal single-dipped domains on $A$.
\end{cor}
Similarly, many other observations in this paper and in the literature may be transcribed from an observation about Arrow's single-peaked domains to an observation about single-dipped domains.\\

The problem of counting maximal Arrow's single-peaked domains has been reduced to counting self-paired maximal Arrow's single-peaked domains, for which we have counted the number with   inherited permutations $(a_1a_2)(sf)$, $(a_1a_{m-2})(sf)$, $(a_2a_{m-2})(sf)$, $(sf)$, and (trivially) for any permutations which has order greater than 2, or which satisfies the conditions of Corollary \ref{cor:SP}. For larger sets, this is a relatively small part of the problem. We have noticed, however, that the number of self-paired domains with $\sigma\neq id$ seems to be relatively small. We predict that the domains with $\sigma=id$ make up at least half of all of the self-paired domains on a given set, and make the following conjecture.

\begin{conj}
Let $\mathcal{N}(m)$ be the number of non-isomorphic maximal Arrow's single-peaked domains on a set $A$ of size $m$. Then the following bound holds:
$$\frac{1}2(\mathcal{P}(m) +\mathcal{P}(m-1))\leq \mathcal{N}(m)\leq \frac{1}2(\mathcal{P}(m) +2\mathcal{P}(m-1)),$$
where $\mathcal{P}(3)=1$ and $\mathcal{P}(m)=2^{m-3}\mathcal{P}(m-1)$ for $m\geq 4$.
\end{conj}

While this paper reduces the problem of counting maximal Arrow's single-peaked domains (and thus maximal single-dipped domains), there is still much yet to be investigated. Furthermore, the problem of counting other types of Condorcet domains is, to the knowledge of the author, still open. 

\end{section}

%%%%%%%%%%%%
%%%%%%%%%%
%%%%%%%%%%%
\begin{section}{Appendix: Domains on a sets of size 5 and 6}
\label{S:append}
In this section we give a representative of each isomorphism class of maximal Arrow's single-peaked domains on sets of size 5 and 6. For compactness, we denote the path $(a_1, a_2, \dots, a_m)$ by $a_1a_2\dots a_m$. The column labelled ``Paths" gives the paths needed to define the given domain.

\begin{longtable}{|p{1.3cm}|p{1.3cm}|p{10cm}|}
\hline
 \multicolumn{3}{|c|}{Domains on a set of size 5} \\
 \hline
Paths&$\sigma$&Linear orders\\
 \hline
$sabcf$, $fcbas$&$(ac)$&$sabcf$, $asbcf$, $bascf$, $abscf$, $cbasf$, $bcasf$, $bacsf$, $abcsf$, $fcbas$, $cfbas$, $cbfas$, $bcfas$, $cbafs$, $bcafs$, $bacfs$, $abcfs$\\
\hline
$sabcf$, $fbcas$&$(abc) $&$sabcf$, $asbcf$, $bascf$, $abscf$, $cbasf$, $bcasf$, $bacsf$, $abcsf$, $fbcas$, $bfcas$, $cbfas$, $bcfas$, $cbafs$, $bcafs$, $bacfs$, $abcfs$\\
\hline
$sabcf$, $facbs$&$(bc) $&$sabcf$, $asbcf$, $bascf$, $abscf$, $bacsf$, $cabsf$, $acbsf$, $abcsf$, $facbs$, $afcbs$, $cafbs$, $acfbs$, $bacfs$, $cabfs$, $acbfs$, $abcfs$\\
\hline
$sabcf$, $fbacs$, $cbasf$&$(ab) $&$sabcf$, $asbcf$, $bascf$, $abscf$, $cbasf$, $bcasf$, $bacsf$, $abcsf$, $fbacs$, $bfacs$, $bafcs$, $abfcs$, $cbafs$, $bcafs$, $bacfs$, $abcfs$\\
\hline
$sabcf$, $fabcs$, $cbasf$&$id$&$sabcf$, $asbcf$, $bascf$, $abscf$, $cbasf$, $bcasf$, $bacsf$, $abcsf$, $fabcs$, $afbcs$, $bafcs$, $abfcs$, $cbafs$, $bcafs$, $bacfs$, $abcfs$\\
\hline
$sabcf$, $fabcs$ $cabsf$&$id$&$sabcf$, $asbcf$, $bascf$, $abscf$, $bacsf$, $cabsf$, $acbsf$, $abcsf$, $fabcs$, $afbcs$, $bafcs$, $abfcs$, $bacfs$, $cabfs$, $acbfs$, $abcfs$\\
\hline
\end{longtable}

Note that the pair $(sabcf,fbacs)$ of extremal paths has two non-equal domains, but they are isomorphic. 

   \begin{longtable} { |p{1.3cm}|p{1.3cm}|p{10cm}|}
 \hline
 \multicolumn{3}{|c|}{Domains on a set of size 6} \\
 \hline
Paths&$\sigma$&Linear orders\\
 \hline
$sabcdf$, $fdcbas$&$(ad)(bc) $&$sabcdf$, $asbcdf$, $bascdf$, $abscdf$, $cbasdf$, $bcasdf$, $bacsdf$, $abcsdf$, $dcbasf$, $cdbasf$, $cbdasf$, $cbadsf$, $bcdasf$, $bcadsf$, $bacdsf$, $abcdsf$, $fdcbas$, $dfcbas$, $dcfbas$, $cdfbas$, $dcbfas$, $cdbfas$, $cbdfas$, $bcdfas$, $dcbafs$, $cdbafs$, $cbdafs$, $cbadfs$, $bcdafs$, $bcadfs$, $bacdfs$, $abcdfs$\\
\hline
$sabcdf$, $fcdbas$&$(acbd) $&$sabcdf$, $asbcdf$, $bascdf$, $abscdf$, $cbasdf$, $bcasdf$, $bacsdf$, $abcsdf$, $dcbasf$, $cdbasf$, $cbdasf$, $cbadsf$, $bcdasf$, $bcadsf$, $bacdsf$, $abcdsf$, $fcdbas$, $cfdbas$, $dcfbas$, $cdfbas$, $dcbfas$, $cdbfas$, $cbdfas$, $bcdfas$, $dcbafs$, $cdbafs$, $cbdafs$, $cbadfs$, $bcdafs$, $bcadfs$, $bacdfs$, $abcdfs$\\
\hline
$sabcdf$, $fdbcas$&$(ad)$&$sabcdf$, $asbcdf$, $bascdf$, $abscdf$, $cbasdf$, $bcasdf$, $bacsdf$, $abcsdf$, $cbdasf$, $cbadsf$, $dbcasf$, $bdcasf$, $bcdasf$, $bcadsf$, $bacdsf$, $abcdsf$, $fdbcas$, $dfbcas$, $dbfcas$, $bdfcas$, $cbdfas$, $dbcfas$, $bdcfas$, $bcdfas$, $cbdafs$, $cbadfs$, $dbcafs$, $bdcafs$, $bcdafs$, $bcadfs$, $bacdfs$, $abcdfs$\\
\hline
$sabcdf$, $fbdcas$&$(abd)$&$sabcdf$, $asbcdf$, $bascdf$, $abscdf$, $cbasdf$, $bcasdf$, $bacsdf$, $abcsdf$, $cbdasf$, $cbadsf$, $dbcasf$, $bdcasf$, $bcdasf$, $bcadsf$, $bacdsf$, $abcdsf$, $fbdcas$, $bfdcas$, $dbfcas$, $bdfcas$, $cbdfas$, $dbcfas$, $bdcfas$, $bcdfas$, $cbdafs$, $cbadfs$, $dbcafs$, $bdcafs$, $bcdafs$, $bcadfs$, $bacdfs$, $abcdfs$\\
\hline
$sabcdf$, $fcdabs$&$(ac)(bd) $&$sabcdf$, $asbcdf$, $bascdf$, $abscdf$, $bacsdf$, $cabsdf$, $acbsdf$, $abcsdf$, $bacdsf$, $dcabsf$, $cdabsf$, $cadbsf$, $cabdsf$, $acdbsf$, $acbdsf$, $abcdsf$, $fcdabs$, $cfdabs$, $dcfabs$, $cdfabs$, $dcafbs$, $cdafbs$, $cadfbs$, $acdfbs$, $bacdfs$, $dcabfs$, $cdabfs$, $cadbfs$, $cabdfs$, $acdbfs$, $acbdfs$, $abcdfs$\\
\hline
$sabcdf$, $fadcbs$&$(bd) $&$sabcdf$, $asbcdf$, $bascdf$, $abscdf$, $bacsdf$, $cabsdf$, $acbsdf$, $abcsdf$, $bacdsf$, $cadbsf$, $cabdsf$, $dacbsf$, $adcbsf$, $acdbsf$, $acbdsf$, $abcdsf$, $fadcbs$, $afdcbs$, $dafcbs$, $adfcbs$, $cadfbs$, $dacfbs$, $adcfbs$, $acdfbs$, $bacdfs$, $cadbfs$, $cabdfs$, $dacbfs$, $adcbfs$, $acdbfs$, $acbdfs$, $abcdfs$\\
\hline
$sabcdf$, $fcbdas$, $dcbasf$&$(acd)$&$
sabcdf$, $asbcdf$, $abcsdf$, $abcdsf$, $abcdfs$, $abscdf$, $cbdasf$, $fcbdas$, $cfbdas$, $cbfdas$, $cbdfas$, $cbdafs$, $bascdf$, $bacsdf$, $bacdsf$, $bacdfs$, $bcdasf$, $bcdfas$, $bcfdas$, $bcdafs$, $cbasdf$, $cbadsf$, $cbadfs$, $bcasdf$, $bcadsf$, $bcadfs$, $dcbasf$, $dcbfas$, $dcbafs$, $cdbasf$, $cdbfas$, $cdbafs$\\
\hline
$sabcdf$, $fcbdas$, $dbcasf$&$(acd)$&$
sabcdf$, $asbcdf$, $abscdf$, $abcsdf$, $abcdsf$, $abcdfs$, $cbdasf$, $fcbdas$, $cbfdas$, $cfbdas$, $cbdfas$, $cbdafs$, $bascdf$, $bacsdf$, $bacdsf$, $bacdfs$, $bcdasf$, $ bcfdas$, $bcdfas$, $bcdafs$, $cbasdf$, $cbadsf$, $cbadfs$, $bcasdf$, $bcadsf$, $bcadfs$, $dbcasf$, $dbcfas$, $dbcafs$, $bdcasf$, $bdcfas$, $bdcafs$\\
\hline
$sabcdf$, $fcbads$, $dcbasf $&$(ac)$&$sabcdf$, $asbcdf$, $abscdf$, $abcsdf$, $abcdsf$, $abcfds$, $abcdfs$, $
cbasdf$, $cbadsf$, $cfbads$, $cbfads$, $cbafds$, $cbadfs$, $fcbads$, $
bascdf$, $bcasdf$, $bacsdf$, $dcbasf$, $cdbasf$, $cbdasf$, $bcdasf$, $bcadsf$, $bacdsf$, $bcfads$, $bcafds$, $bacfds$, $dcbafs$, $cdbafs$, $cbdafs$, $bcdafs$, $bcadfs$, $bacdfs$\\
\hline
$sabcdf$, $fcbads$, $dbcasf $&$(ac)$&$sabcdf$, $asbcdf$, $abscdf$, $abcsdf$, $abcdsf$, $abcfds$, $abcdfs$, $
cbasdf$, $cbadsf$, $fcbads$, $cfbads$, $cbfads$, $cbafds$, $cbadfs$, $
 bascdf$, $bcasdf$, $bacsdf$, $cbdasf$, $dbcasf$, $bdcasf$, $bcdasf$, $bcadsf$, $bacdsf$, $bcfads$, $ bcafds$, $bacfds$, $cbdafs$, $dbcafs$, $bdcafs$, $bcdafs$, $bcadfs$, $bacdfs$\\
\hline
$sabcdf$, $fbdacs$, $cbadsf $&$(abdc)$&$sabcdf$, $asbcdf$, $bascdf$, $abscdf$, $cbasdf$, $bcasdf$, $bacsdf$, $abcsdf$, $cbadsf$, $bcadsf$, $dbacsf$, $bdacsf$, $badcsf$, $bacdsf$, $abdcsf$, $abcdsf$, $fbdacs$, $bfdacs$, $dbfacs$, $bdfacs$, $dbafcs$, $bdafcs$, $badfcs$, $abdfcs$, $cbadfs$, $bcadfs$, $dbacfs$, $bdacfs$, $badcfs$, $bacdfs$, $abdcfs$, $abcdfs$\\
\hline
$sabcdf$, $fbdacs$, $cabdsf$&$(abdc)$&$sabcdf$, $asbcdf$, $bascdf$, $abscdf$, $bacsdf$, $cabsdf$, $acbsdf$, $abcsdf$, $dbacsf$, $bdacsf$, $badcsf$, $bacdsf$, $cabdsf$, $acbdsf$, $abdcsf$, $abcdsf$, $fbdacs$, $bfdacs$, $dbfacs$, $bdfacs$, $dbafcs$, $bdafcs$, $badfcs$, $abdfcs$, $dbacfs$, $bdacfs$, $badcfs$, $bacdfs$, $cabdfs$, $acbdfs$, $abdcfs$, $abcdfs$\\
\hline
$sabcdf$, $fbcdas$, $dcbasf$&$(abcd)$&$
sabcdf$, $asbcdf$, $abscdf$, $abcdsf$, $abcdfs$, $abcsdf$, $
bcdasf$, $fbcdas$, $bfcdas$, $bcfdas$, $bcdfas$, $bcdafs$, $
bascdf$, $cbasdf$, $bcasdf$, $bacsdf$, $cbfdas$, $dcbfas$, $cdbfas$, $cbdfas$, $
dcbasf$, $cdbasf$, $cbdasf$, $cbadsf$, $bcadsf$, $bacdsf$, $
dcbafs$, $cdbafs$, $cbdafs$, $cbadfs$, $bcadfs$, $bacdfs$\\
\hline
$sabcdf$, $fbcdas$, $dbcasf $&$(abcd)$&$
sabcdf$, $asbcdf$, $abscdf$, $abcdfs$, $abcdsf$, $abcsdf$, $
bcdasf$, $fbcdas$, $bfcdas$, $bcdfas$, $bcdafs$, $bcfdas$, $
bascdf$, $cbasdf$, $bcasdf$, $bacsdf$, $cbfdas$, $cbdfas$, $dbcfas$, $bdcfas$, $
cbdasf$, $cbadsf$, $dbcasf$, $bdcasf$, $bcadsf$, $bacdsf$, $
cbdafs$, $cbadfs$, $dbcafs$, $bdcafs$, $bcadfs$, $bacdfs$\\
\hline
$sabcdf$, $facdbs$, $dcabsf $&$(bcd)$&$
sabcdf$, $asbcdf$, $bascdf$, $abscdf$, $bacsdf$, $cabsdf$, $acbsdf$, $abcsdf$, $bacdsf$, $dcabsf$, $cdabsf$, $cadbsf$, $cabdsf$, $acdbsf$, $acbdsf$, $abcdsf$, $facdbs$, $afcdbs$, $cafdbs$, $acfdbs$, $dcafbs$, $cdafbs$, $cadfbs$, $acdfbs$, $bacdfs$, $dcabfs$, $cdabfs$, $cadbfs$, $cabdfs$, $acdbfs$, $acbdfs$, $abcdfs$\\
\hline
$sabcdf$, $facdbs$, $dacbsf$&$(bcd)$&$
sabcdf$, $asbcdf$, $bascdf$, $abscdf$, $bacsdf$, $cabsdf$, $acbsdf$, $abcsdf$, $bacdsf$, $cadbsf$, $cabdsf$, $dacbsf$, $adcbsf$, $acdbsf$, $acbdsf$, $abcdsf$, $facdbs$, $afcdbs$, $cafdbs$, $acfdbs$, $cadfbs$, $dacfbs$, $adcfbs$, $acdfbs$, $bacdfs$, $cadbfs$, $cabdfs$, $dacbfs$, $adcbfs$, $acdbfs$, $acbdfs$, $abcdfs$\\
\hline
$sabcdf$, $facbds$, $dbacsf$&$(bc)$&$
sabcdf$, $asbcdf$, $bascdf$, $abscdf$, $bacsdf$, $cabsdf$, $acbsdf$, $abcsdf$, $dbacsf$, $bdacsf$, $badcsf$, $bacdsf$, $cabdsf$, $acbdsf$, $abdcsf$, $abcdsf$, $facbds$, $afcbds$, $cafbds$, $acfbds$, $bacfds$, $cabfds$, $acbfds$, $abcfds$, $dbacfs$, $bdacfs$, $badcfs$, $bacdfs$, $cabdfs$, $acbdfs$, $abdcfs$, $abcdfs$\\
\hline
$sabcdf$, $facbds$, $dacbsf$&$(bc)$&$
sabcdf$, $asbcdf$, $bascdf$, $abscdf$, $bacsdf$, $cabsdf$, $acbsdf$, $abcsdf$, $bacdsf$, $cadbsf$, $cabdsf$, $dacbsf$, $adcbsf$, $acdbsf$, $acbdsf$, $abcdsf$, $facbds$, $afcbds$, $cafbds$, $acfbds$, $bacfds$, $cabfds$, $acbfds$, $abcfds$, $bacdfs$, $cadbfs$, $cabdfs$, $dacbfs$, $adcbfs$, $acdbfs$, $acbdfs$, $abcdfs$\\
\hline
$sabcdf$, $fabdcs$, $cbadsf$, $dbacsf $&$(dc)$&$
sabcdf$, $asbcdf$, $bascdf$, $abscdf$, $cbasdf$, $bcasdf$, $bacsdf$, $abcsdf$, $cbadsf$, $bcadsf$, $dbacsf$, $bdacsf$, $badcsf$, $bacdsf$, $abdcsf$, $abcdsf$, $fabdcs$, $afbdcs$, $bafdcs$, $abfdcs$, $dbafcs$, $bdafcs$, $badfcs$, $abdfcs$, $cbadfs$, $bcadfs$, $dbacfs$, $bdacfs$, $badcfs$, $bacdfs$, $abdcfs$, $abcdfs$\\
\hline
$sabcdf$, $fabdcs$, $cbadsf$, $dabcsf $&$(dc)$&$
sabcdf$, $asbcdf$, $bascdf$, $abscdf$, $cbasdf$, $bcasdf$, $bacsdf$, $abcsdf$, $cbadsf$, $bcadsf$, $badcsf$, $bacdsf$, $dabcsf$, $adbcsf$, $abdcsf$, $abcdsf$, $fabdcs$, $afbdcs$, $bafdcs$, $abfdcs$, $badfcs$, $dabfcs$, $adbfcs$, $abdfcs$, $cbadfs$, $bcadfs$, $badcfs$, $bacdfs$, $dabcfs$, $adbcfs$, $abdcfs$, $abcdfs$\\
\hline
$sabcdf$, $fabdcs$, $cabdsf$, $dabcsf $&$(dc)$&$
sabcdf$, $asbcdf$, $bascdf$, $abscdf$, $bacsdf$, $cabsdf$, $acbsdf$, $abcsdf$, $badcsf$, $bacdsf$, $cabdsf$, $acbdsf$, $dabcsf$, $adbcsf$, $abdcsf$, $abcdsf$, $fabdcs$, $afbdcs$, $bafdcs$, $abfdcs$, $badfcs$, $dabfcs$, $adbfcs$, $abdfcs$, $badcfs$, $bacdfs$, $cabdfs$, $acbdfs$, $dabcfs$, $adbcfs$, $abdcfs$, $abcdfs$\\
\hline
$sabcdf$, $fbadcs$, $cbadsf$, $dbacsf$&$(ab)(dc)$&$
sabcdf$, $asbcdf$, $bascdf$, $abscdf$, $cbasdf$, $bcasdf$, $bacsdf$, $abcsdf$, $cbadsf$, $bcadsf$, $dbacsf$, $bdacsf$, $badcsf$, $bacdsf$, $abdcsf$, $abcdsf$, $fbadcs$, $bfadcs$, $bafdcs$, $abfdcs$, $dbafcs$, $bdafcs$, $badfcs$, $abdfcs$, $cbadfs$, $bcadfs$, $dbacfs$, $bdacfs$, $badcfs$, $bacdfs$, $abdcfs$, $abcdfs$\\
\hline
$sabcdf$, $fbadcs$, $cbadsf$, $dabcsf $&$(ab)(dc)$&$
sabcdf$, $asbcdf$, $bascdf$, $abscdf$, $cbasdf$, $bcasdf$, $bacsdf$, $abcsdf$, $cbadsf$, $bcadsf$, $badcsf$, $bacdsf$, $dabcsf$, $adbcsf$, $abdcsf$, $abcdsf$, $fbadcs$, $bfadcs$, $bafdcs$, $abfdcs$, $badfcs$, $dabfcs$, $adbfcs$, $abdfcs$, $cbadfs$, $bcadfs$, $badcfs$, $bacdfs$, $dabcfs$, $adbcfs$, $abdcfs$, $abcdfs$\\
\hline
$sabcdf$, $fbadcs$, $dbacsf$, $cabdsf $&$(ab)(dc)$&$
sabcdf$, $asbcdf$, $bascdf$, $abscdf$, $bacsdf$, $cabsdf$, $acbsdf$, $abcsdf$, $dbacsf$, $bdacsf$, $badcsf$, $bacdsf$, $cabdsf$, $acbdsf$, $abdcsf$, $abcdsf$, $fbadcs$, $bfadcs$, $bafdcs$, $abfdcs$, $dbafcs$, $bdafcs$, $badfcs$, $abdfcs$, $dbacfs$, $bdacfs$, $badcfs$, $bacdfs$, $cabdfs$, $acbdfs$, $abdcfs$, $abcdfs$\\
\hline
$sabcdf$, $fbacds$, $dcbasf $&$(ab)$&$
sabcdf$, $asbcdf$, $abscdf$, $abcsdf$, $abfcds$, $abcfds$, $
bascdf$, $bacsdf$, $fbacds$, $bfacds$, $bafcds$, $bacfds$, $ 
cbasdf$, $bcasdf$, $cbafds$, $bcafds$, $
dcbasf$, $cdbasf$, $cbdasf$, $cbadsf$, $bcdasf$, $bcadsf$, $bacdsf$, $abcdsf$, $
dcbafs$, $cdbafs$, $cbdafs$, $cbadfs$, $bcdafs$, $bcadfs$, $bacdfs$, $abcdfs$\\
\hline
$sabcdf$, $fbacds$, $dbcasf$&$(ab)$&$
sabcdf$, $asbcdf$, $abcsdf$, $abscdf$, $abfcds$, $abcfds$, $
bascdf$, $bacsdf$, $fbacds$, $bfacds$, $bafcds$, $bacfds$, $
cbasdf$, $bcasdf$, $cbafds$, $bcafds$, $
cbdasf$, $cbadsf$, $dbcasf$, $bdcasf$, $bcdasf$, $bcadsf$, $bacdsf$, $abcdsf$, $
cbdafs$, $cbadfs$, $dbcafs$, $bdcafs$, $bcdafs$, $bcadfs$, $bacdfs$, $abcdfs$\\
\hline
$sabcdf$, $fbacds$, $cbadsf$, $dbacsf$&$(ab)$&$
sabcdf$, $asbcdf$, $abscdf$, $abcsdf$, $abfcds$, $abcfds$, $
bascdf$, $bacsdf$, $fbacds$, $bfacds$, $bafcds$, $bacfds$, $
cbasdf$, $bcasdf$, $cbafds$, $bcafds$, $
cbadsf$, $bcadsf$, $dbacsf$, $bdacsf$, $badcsf$, $bacdsf$, $abdcsf$, $abcdsf$, $
cbadfs$, $bcadfs$, $dbacfs$, $bdacfs$, $badcfs$, $bacdfs$, $abdcfs$, $abcdfs$\\
\hline
$sabcdf$, $fbacds$, $cbadsf$, $dabcsf$&$(ab)$&$
sabcdf$, $asbcdf$, $abscdf$, $abcsdf$, $abfcds$, $abcfds$, $
bascdf$, $bacsdf$, $fbacds$, $bfacds$, $bafcds$, $bacfds$, $
 cbasdf$, $bcasdf$, $ cbafds$, $bcafds$, $
cbadsf$, $bcadsf$, $badcsf$, $bacdsf$, $dabcsf$, $adbcsf$, $abdcsf$, $abcdsf$, $
cbadfs$, $bcadfs$, $badcfs$, $bacdfs$, $dabcfs$, $adbcfs$, $abdcfs$, $abcdfs$\\
\hline
$sabcdf$, $fbcads$, $dcbasf $&$(abc)$&$
sabcdf$, $asbcdf$, $abscdf$, $abcsdf$, $abcfds$, $
bcasdf$, $fbcads$, $bfcads$, $bcfads$, $bcafds$, $
bascdf$, $cbasdf$, $bacsdf$, $cbfads$, $cbafds$, $bacfds$, $
dcbasf$, $cdbasf$, $cbdasf$, $cbadsf$, $bcdasf$, $bcadsf$, $bacdsf$, $abcdsf$, $
dcbafs$, $cdbafs$, $cbdafs$, $cbadfs$, $bcdafs$, $bcadfs$, $bacdfs$, $abcdfs$\\
\hline
$sabcdf$, $fbcads$, $dbcasf $&$(abc)$&$
sabcdf$, $asbcdf$, $abscdf$, $abcsdf$, $abcfds$, $
bcasdf$, $fbcads$, $bfcads$, $bcfads$, $bcafds$, $
bascdf$, $cbasdf$, $bacsdf$, $cbfads$, $cbafds$, $bacfds$, $
cbdasf$, $cbadsf$, $dbcasf$, $bdcasf$, $bcdasf$, $bcadsf$, $bacdsf$, $abcdsf$, $
cbdafs$, $cbadfs$, $dbcafs$, $bdcafs$, $bcdafs$, $bcadfs$, $bacdfs$, $abcdfs$\\
\hline
$sabcdf$, $fbcads$, $dbacsf $&$(abc)$&$
sabcdf$, $asbcdf$, $abscdf$, $abcsdf$, $abcfds$, $
bcasdf$, $fbcads$, $bfcads$, $bcfads$, $bcafds$, $
bascdf$, $cbasdf$, $bacsdf$, $cbfads$, $cbafds$, $bacfds$, $
cbadsf$, $bcadsf$, $dbacsf$, $bdacsf$, $badcsf$, $bacdsf$, $abdcsf$, $abcdsf$, $
cbadfs$, $bcadfs$, $dbacfs$, $bdacfs$, $badcfs$, $bacdfs$, $abdcfs$, $abcdfs$\\
\hline
$sabcdf$, $fbcads$, $dabcsf$&$(abc)$&$
sabcdf$, $asbcdf$, $abscdf$, $abcsdf$, $abcfds$, $
bcasdf$, $fbcads$, $bfcads$, $bcfads$, $bcafds$, $
bascdf$, $cbasdf$, $bacsdf$, $cbfads$, $cbafds$, $bacfds$, $
cbadsf$, $bcadsf$, $badcsf$, $bacdsf$, $dabcsf$, $adbcsf$, $abdcsf$, $abcdsf$, $
cbadfs$, $bcadfs$, $badcfs$, $bacdfs$, $dabcfs$, $adbcfs$, $abdcfs$, $abcdfs$\\
\hline
$sabcdf$, $fabcds$, $dcbasf$&$id$&$sabcdf$, $asbcdf$, $bascdf$, $abscdf$, $cbasdf$, $bcasdf$, $bacsdf$, $abcsdf$, $dcbasf$, $cdbasf$, $cbdasf$, $cbadsf$, $bcdasf$, $bcadsf$, $bacdsf$, $abcdsf$, $fabcds$, $afbcds$, $bafcds$, $abfcds$, $cbafds$, $bcafds$, $bacfds$, $abcfds$, $dcbafs$, $cdbafs$, $cbdafs$, $cbadfs$, $bcdafs$, $bcadfs$, $bacdfs$, $abcdfs$\\
\hline
$sabcdf$, $fabcds$, $dbcasf$&$id$&$sabcdf$, $asbcdf$, $bascdf$, $abscdf$, $cbasdf$, $bcasdf$, $bacsdf$, $abcsdf$, $cbdasf$, $cbadsf$, $dbcasf$, $bdcasf$, $bcdasf$, $bcadsf$, $bacdsf$, $abcdsf$, $fabcds$, $afbcds$, $bafcds$, $abfcds$, $cbafds$, $bcafds$, $bacfds$, $abcfds$, $cbdafs$, $cbadfs$, $dbcafs$, $bdcafs$, $bcdafs$, $bcadfs$, $bacdfs$, $abcdfs$\\
\hline
$sabcdf$, $fabcds$, $dcabsf $&$id$&$sabcdf$, $asbcdf$, $bascdf$, $abscdf$, $bacsdf$, $cabsdf$, $acbsdf$, $abcsdf$, $bacdsf$, $dcabsf$, $cdabsf$, $cadbsf$, $cabdsf$, $acdbsf$, $acbdsf$, $abcdsf$, $fabcds$, $afbcds$, $bafcds$, $abfcds$, $bacfds$, $cabfds$, $acbfds$, $abcfds$, $bacdfs$, $dcabfs$, $cdabfs$, $cadbfs$, $cabdfs$, $acdbfs$, $acbdfs$, $abcdfs$\\
\hline
$sabcdf$, $fabcds$, $dacbsf$&$id$&$sabcdf$, $asbcdf$, $bascdf$, $abscdf$, $bacsdf$, $cabsdf$, $acbsdf$, $abcsdf$, $bacdsf$, $cadbsf$, $cabdsf$, $dacbsf$, $adcbsf$, $acdbsf$, $acbdsf$, $abcdsf$, $fabcds$, $afbcds$, $bafcds$, $abfcds$, $bacfds$, $cabfds$, $acbfds$, $abcfds$, $bacdfs$, $cadbfs$, $cabdfs$, $dacbfs$, $adcbfs$, $acdbfs$, $acbdfs$, $abcdfs$\\
\hline
$sabcdf$, $fabcds$, $cbadsf$, $dbacsf$&$id$&$sabcdf$, $asbcdf$, $bascdf$, $abscdf$, $cbasdf$, $bcasdf$, $bacsdf$, $abcsdf$, $cbadsf$, $bcadsf$, $dbacsf$, $bdacsf$, $badcsf$, $bacdsf$, $abdcsf$, $abcdsf$, $fabcds$, $afbcds$, $bafcds$, $abfcds$, $cbafds$, $bcafds$, $bacfds$, $abcfds$, $cbadfs$, $bcadfs$, $dbacfs$, $bdacfs$, $badcfs$, $bacdfs$, $abdcfs$, $abcdfs$\\
\hline
$sabcdf$, $fabcds$, $cbadsf$, $dabcsf$&$id$&$sabcdf$, $asbcdf$, $bascdf$, $abscdf$, $cbasdf$, $bcasdf$, $bacsdf$, $abcsdf$, $cbadsf$, $bcadsf$, $badcsf$, $bacdsf$, $dabcsf$, $adbcsf$, $abdcsf$, $abcdsf$, $fabcds$, $afbcds$, $bafcds$, $abfcds$, $cbafds$, $bcafds$, $bacfds$, $abcfds$, $cbadfs$, $bcadfs$, $badcfs$, $bacdfs$, $dabcfs$, $adbcfs$, $abdcfs$, $abcdfs$\\
\hline
$sabcdf$, $fabcds$, $dbacsf$, $cabdsf$&$id$&$sabcdf$, $asbcdf$, $bascdf$, $abscdf$, $bacsdf$, $cabsdf$, $acbsdf$, $abcsdf$, $dbacsf$, $bdacsf$, $badcsf$, $bacdsf$, $cabdsf$, $acbdsf$, $abdcsf$, $abcdsf$, $fabcds$, $afbcds$, $bafcds$, $abfcds$, $bacfds$, $cabfds$, $acbfds$, $abcfds$, $dbacfs$, $bdacfs$, $badcfs$, $bacdfs$, $cabdfs$, $acbdfs$, $abdcfs$, $abcdfs$\\
\hline
$sabcdf$, $fabcds$, $cabdsf$, $dabcsf $&$id$&$sabcdf$, $asbcdf$, $bascdf$, $abscdf$, $bacsdf$, $cabsdf$, $acbsdf$, $abcsdf$, $badcsf$, $bacdsf$, $cabdsf$, $acbdsf$, $dabcsf$, $adbcsf$, $abdcsf$, $abcdsf$, $fabcds$, $afbcds$, $bafcds$, $abfcds$, $bacfds$, $cabfds$, $acbfds$, $abcfds$, $badcfs$, $bacdfs$, $cabdfs$, $acbdfs$, $dabcfs$, $adbcfs$, $abdcfs$, $abcdfs$\\
\hline
\end{longtable}

\end{section}

\vfill
\begin{section}{References}
[1] K. J. Arrow. {\it Social choice and individual values}. Wiley, New York, 1st
edition, 1951.\\

[2] D. Black. {\it On the rationale of group decision-making}. Journal of Political Economy, 56:23–34, 1948.\\

[3] M. J. Condorcet et al. {\it Essai sur l'application de l'analyse {\`a} la probabilit{\'e} des d{\'e}cisions rendues {\`a} la pluralit{\'e} des voix}, volume 252. American
Mathematical Soc., 1972.\\

[4] V.I. Danilov and G. A. Koshevoy. {\it Maximal condorcet domains}. Order,
30(1):181–194, 2013.\\

[5] P.C. Fishburn. {\it Acyclic sets of linear orders: A progress report}. Social
Choice and Welfare, 19(2):431–447, 2002.\\

[6] {\'A} Galambos and V Reiner. {\it Acyclic sets of linear orders via the Bruhat
orders}. Social Choice and Welfare, 30(2):245–264, 2008.\\

[7] B. Monjardet. {\it Acyclic domains of linear orders: A survey}. In S. Brams,
W. Gehrlein, and F. Roberts, editors, The Mathematics of Preference, Choice and Order, Studies in Choice and Welfare, page 139–160. Springer, 2009.\\

[8] C. Puppe. {\it The single-peaked domain revisited: A simple global characterization}. Journal of Economic Theory, 176:55 – 80, 2018.\\

[9] C. Puppe and A. Slinko. {\it Condorcet domains, median graphs and the
single crossing property}. Economic Theory, pages 1–34, 2016.\\

[10] D Romero. {\it Variations sur l’effet Condorcet}. PhD thesis, Universit{\'e}
Joseph-Fourier-Grenoble I, 1978.\\

[11] A.K Sen. {\it A possibility theorem on majority decisions}. Econometrica,
34:491–499, 1966.\\

[12] A. Slinko. {\it Arrow’s single-peaked Condorcet domains}. University of
Auckland Algebra and Combinatorics Seminar, 2019.\\

[13] A. Slinko. {\it Condorcet domains satisfying Arrow’s single-peakedness}.
Journal of Mathematical Economics, 2019.
\end{section}

\end{document}